\documentclass[12pt]{amsart}

\title{Preperiodic points of polynomial dynamical systems over finite fields}
\author{Aaron Andersen\and Derek Garton}
\address{Fariborz Maseeh Department of Mathematics and Statistics, Portland State University}
\email{\href{mailto:aaander2@pdx.edu}{aaander2@pdx.edu},\ \href{mailto:gartondw@pdx.edu}{gartondw@pdx.edu}}
\date{\today}

\makeatletter
\@namedef{subjclassname@2020}{\textup{2020} Mathematics Subject Classification}
\makeatother



\usepackage{amssymb,amsmath,url,MnSymbol,colonequals,tikz-cd}

\usepackage{hyperref}
\hypersetup{breaklinks=true,
colorlinks=true,
urlcolor=blue,
linkcolor=cyan
}

\usepackage{tikz}

\usepackage[inline,shortlabels]{enumitem}

\usepackage[top=1in,bottom=1in,left=1in,right=1in]{geometry}

\usepackage[capitalise,nameinlink]{cleveref}

\newcommand{\Z}{\ensuremath{\mathbb{Z}}}

\renewcommand{\P}{\ensuremath{\mathbb{P}}}

\newcommand{\R}{\ensuremath{\mathbb{R}}}
\newcommand{\F}{\ensuremath{\mathbb{F}}}

\newcommand{\lv}{\ensuremath{\left\vert}}
\newcommand{\rv}{\ensuremath{\right\vert}}

\newcommand{\lp}{\ensuremath{\left(}}
\newcommand{\rp}{\ensuremath{\right)}}
\newcommand{\lb}{\ensuremath{\left\{}}
\newcommand{\rb}{\ensuremath{\right\}}}

\DeclareMathOperator{\Gal}{Gal}

\DeclareMathOperator{\Per}{Per}

\DeclareMathOperator{\id}{id}
\DeclareMathOperator{\Frac}{Frac}
\DeclareMathOperator{\Spec}{Spec}

\DeclareMathOperator{\fix}{f}

\theoremstyle{plain}
\newtheorem{theorem}{Theorem}[section]
\newtheorem{lemma}[theorem]{Lemma}
\newtheorem{corollary}[theorem]{Corollary}
\newtheorem{proposition}[theorem]{Proposition}

\newtheorem*{namedthm}{\namedthmname}
\newcounter{namedthm}

\theoremstyle{remark}
\newtheorem{remark}[theorem]{Remark}

\theoremstyle{definition}



\begin{document}
\begin{abstract}
For a prime $p$, positive integers $r,n$, and a polynomial $f$ with coefficients in $\F_{p^r}$, let $W_{p,r,n}(f)=f^n\lp \F_{p^r}\rp\setminus f^{n+1}\lp\F_{p^r}\rp$.
As $n$ varies, the $W_{p,r,n}(f)$ partition the set of strictly preperiodic points of the dynamical system induced by the action of $f$ on $\F_{p^r}$.
In this paper we compute statistics of strictly preperiodic points of dynamical systems induced by unicritical polynomials over finite fields by obtaining effective upper bounds for the proportion of $\F_{p^r}$ lying in a given $W_{p,r,n}(f)$.
Moreover, when we generalize our definition of $W_{p,r,n}(f)$, we obtain both upper and lower bounds for the resulting averages.
\end{abstract}

\maketitle
\tableofcontents

\section{Introduction}
\label{intro}

A \emph{\textup{(}discrete\textup{)} dynamical system} is a pair $\lp S,f\rp$ consisting of a set $S$ and a function $f\colon S\to S$.
For notational convenience, for any positive integer $n$, we let $f^n=\overbrace{f\circ\cdots\circ f}^{n\text{ times}}$; furthermore, we set $f^0=\id_S$. 
For any $s\in S$, if there is some positive integer $n$ such that $f^n(s)=s$, we say that $s$ is \emph{periodic} (for $f$).
Let $\Per{\lp S,f\rp}=\lb s\in S\mid s\text{ is periodic for }f\rb$.

When $S$ is a finite field, say $S=\F_q$ for some prime power $q$, and $f$ is a polynomial with coefficients in $\F_q$, a question arises: for $n\in\Z_{\geq0}$, what is the size of $f^n\lp\F_q\rp$?
This question has been studied, for example, in \cite{JKMT,HB,Juul,JuulP,G,GA}.
In each of these papers, the authors use the answers they find to address the related question: what is the size of $\Per{\lp\F_q,f\rp}$?
This is due to the fact that for any $n\in\Z_{\geq0}$, the set $f^n\lp\F_q\rp$ contains $\Per{\lp\F_q,f\rp}$---see~\cite[Lemma~5.2]{JKMT}.
Specifically, upper bounds on the size of $f^n\lp\F_q\rp$ yield upper bounds on the size  of $\Per{\lp\F_q,f\rp}$.

In this paper, we turn to the study of strictly preperiodic points.
If $(S,f)$ is a dynamical system and $s\in S$, we say that $s$ is \emph{strictly preperiodic} (for $f$) if $s$ is not periodic and there is some positive integer $n$ such that $f^n(s)$ is periodic.
Of course, when $S$ is finite, the strictly preperiodic points are precisely $S\setminus\Per{\lp S,f\rp}$.
In the finite case, we partition the strictly preperiodic points as follows: for a nonnegative integer $n$, let
\[
W_n\lp S , f \rp= f^n \lp S \rp\setminus f^{n+1}\lp S \rp.
\]
We prove in \cref{aaronlemma} that the nonempty $W_n\lp S , f \rp$ do indeed partition the strictly preperiodic points of $(S,f)$; see \cref{prettypicture} for an illustration of this phenomenon.
The purpose of this paper is to average the proportion of $S$ in these $W_n\lp S , f \rp$, as $f$ varies; so when $S$ is finite, let
\[
w_n\lp S , f \rp
=\frac{\lv W_n(S,f)\rv}
{\lv S\rv}.
\]
There is a natural generalization of this classification of strictly preperiodic points: for a dynamical system $(S,f)$ and integers $m,n$ with $n>m\geq0$, we define
\[
W_{m,n}\lp S,f \rp
=f^m(S)\setminus f^n(S).
\]
As above, when $S$ is finite, we write $w_{m,n}\lp S , f \rp=\lv W_{m,n}\lp S , f \rp\rv\cdot\lv S\rv^{-1}$.
Of course, it is clear from these definitions that $W_n\lp S , f \rp=W_{n,n+1}\lp S , f \rp$.

Before stating our results, we introduce one more bit of notation.
If $q$ is a prime power, $d\in\Z_{\geq2}$, and $\alpha\in \F_q$, we will write $f_{d,\alpha}=f_{d,\alpha}(x)=x^d+\alpha\in\F_q[x]$.
As these polynomials have only one critical point, they are examples of \emph{unicritical polynomials}; our main results hold for dynamical systems induced by such polynomials.
In \cref{upboundz}, we prove \cref{quadcor}, which is the $d=2$ case of the more-general \cref{upperboundforw}. 

\begin{corollary}\label{quadcor}
Suppose $p>3$ is prime.
Choose positive integers $r,n$ with $n>2$ and $\alpha\in\F_{p^r}$ with $\F_p\lp \alpha \rp = \F_{p^r}$.
If $r>2^{2n+3}$, then 
\[
w_n\lp \F_{p^r},f_{2,\alpha} \rp
<15\lp\frac{\log{n}}{n^2}\rp +\frac{32}{p^{r/2}}.
\]
\end{corollary}
Unlike previous work, we also obtain \emph{lower} bounds.
The work on periodic proportions previously mentioned uses only upper bounds on image size; \cref{aaroncorollary} follows from using both upper and lower bounds on image size (which we record in \cref{technical}).

\begin{corollary}\label{aaroncorollary}
Let $d\in\Z_{\geq2}$, and suppose $p$ is a prime satisfying $p>(d!)^2$ and $p\equiv1\pmod{d}$.
Choose $r,m,n\in\Z_{\geq1}$ with $5<m<n$, and $\alpha\in\F_{p^r}$ with $\F_p\lp \alpha \rp = \F_{p^r}$.
If $r>2d^{2n}$, then
\[
\frac{7}{8(d-1)}\lp\frac{1}{m}-\frac{1}{n}-\frac{4\log{m}}{mn}\rp-\frac{16d}{p^{r/2}}
<w_{m,n}(\F_{p^r},f_{d,\alpha})
<\frac{2}{d-1} \lp \frac{1}{m}-\frac{1}{n}
+\frac{4\log{n}}{mn} \rp +\frac{16d}{p^{r/2}}.
\]
\end{corollary}

In \cref{averaging}, we compute upper bounds on the statistics of strictly preperiodic points, averaging over all quadratic polynomials.
To do so, we use that fact that any quadratic polynomial (in odd characteristic) is conjugate to a unicritical polynomial.

\begin{theorem}\label{averageupperbound}
Suppose $p>3$ is prime.
Let $n,r\in\Z_{\geq 1}$.
If $n>133$ and $r>2^{2n+3}$, then
\[
\frac{1}
{\lv\lb f\in\F_{p^r}[x]
\mid\deg{f}=2\rb\rv}
\cdot\sum_{\substack{f\in\F_{p^r}[x]\\
\deg{f}=2}}
{w_n\lp \F_{p^r}, f\rp}
< \frac{1}{n^{3/2}}+\frac{34}{p^{r/2}}.
\]
\end{theorem}
\noindent Moreover, as in \cref{aaroncorollary}, we can obtain both lower and upper bounds for statistics of strictly preperiodic points by the using lower bounds on image sizes given in \cref{technical}.

\begin{corollary}\label{genheadintro}
Suppose $p>3$ is prime.
Let $r,m,n\in\Z_{\geq 1}$ with $5<m<n$.
If $r>2^{2n+1}$, then
\begin{align*}
\frac{7}{8}\lp\frac{1}{m} - \frac{1}{n} \rp-4\lp\frac{\log m}{mn}\rp
<\frac{1}
{\lv\lb f\in\F_{p^r}[x]
\mid\deg{f}=2\rb\rv}
\cdot\sum_{\substack{f\in\F_{p^r}[x]\\
\deg{f}=2}}
&{w_{m,n}\lp \F_{p^r}, f\rp}\\
&\qquad<2\lp\frac{1}{m} - \frac{1}{n}\rp
+9\lp\frac{\log n}{mn}\rp.
\end{align*}
\end{corollary}

The organization of this paper is as follows.
In \cref{laries}, we prove basic facts about our partition of strictly preperiodic points, as well as the main technical tool needed for our applications, \cref{technical}, which gives an effective estimate of image sizes of polynomial dynamical systems.
In \cref{upboundz} and \cref{theorems}, we use \cref{technical} to prove the upper and lower bounds in \cref{aaroncorollary}, respectively.
Finally, in \cref{averaging}, we compute averages over all quadratic polynomials.

Before proceeding to \cref{laries}, we prove \cref{aaronlemma}.

\begin{lemma}\label{aaronlemma}
If $(S,f)$ is a dynamical system and $S$ is finite, then
\[
\lb W_n\lp S , f \rp\mid n\in\Z_{\geq0}\text{ and }W_n\lp S , f \rp\neq\emptyset\rb
\]
is a partition of the strictly preperiodic points of $(S,f)$.
\end{lemma}
\begin{proof}
We begin by showing that the the sets $W_n(S,f)$ contain all strictly preperiodic points of $(S,f)$.
To this end, choose any strictly preperiodic point $s_0\in S$ and set
\[
P_{s_0}=\lb s\in S\mid\text{there exists }n\in\Z_{\geq0}\text{ such that }f^n(s)=s_0\rb.
\]
We claim that for any $s\in P_{s_0}$, there is a \emph{unique} $n\in\Z_{>0}$ such that $f^n(s)=s_0$.
Indeed, this follows from the fact that $s$ is not periodic.
For any $s\in P_{s_0}$, let $n_s$ be this positive integer.
Since $S$ is finite, we may set $n_0=\max{\lp\lb n_s\mid s\in P_{s_0}\rb\rp}$.
Then $s_0\in W_{n_0}(S,f)$.

To see that the $W_n(S,f)$ are pairwise disjoint, choose any $m,n\in\Z_{\geq0}$ with $n>m$.
Since $f^n(S)\subseteq f^m(S)$ and $f^{n+1}(S)\subseteq f^{m+1}(S)$, we see that
\[
W_m(S,f)\cap W_n(S,f)
=f^n(S)\setminus f^{m+1}(S)
=\emptyset.
\]
\end{proof}

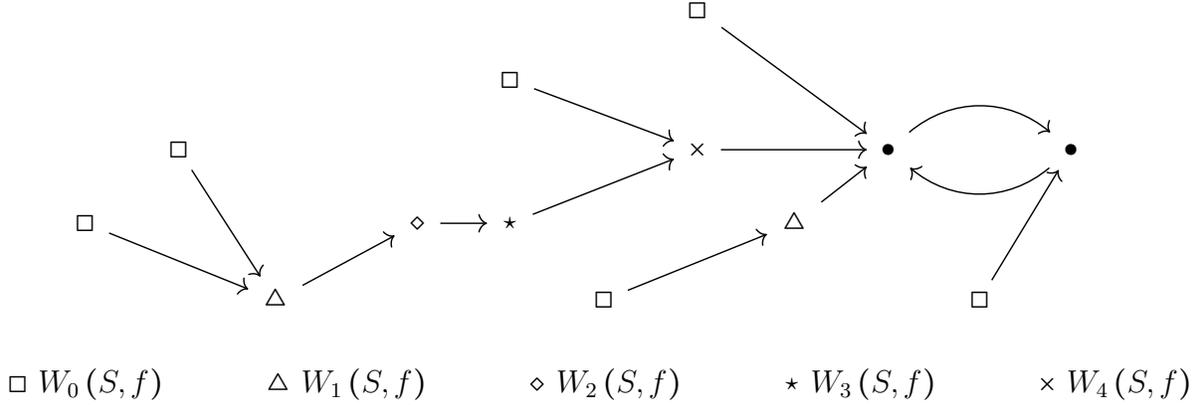
\begin{figure}
\caption{A partition of the strictly preperiodic points of a dynamical system $(S,f)$}\label{prettypicture}
\begin{center}
\begin{tikzcd}[row sep = 1.0 em,column sep = 1.5 em]
& & & & & & & \square \arrow[ddrr] & & & & & \\
& & & & & \square \arrow[drr] & & & & & & & \\
& \square \arrow[ddr] & & & &  & & \times \arrow[rr] & & \bullet \arrow[rr, bend left=40] & & \bullet \arrow[ll, bend left = 40] \\
\square \arrow[drr] & & & & \diamond \arrow[r] & \star \arrow[rru] & &  & \triangle\arrow[ru] &  & & & \\
& & \triangle \arrow[rru] & & & & \square\arrow[rru] &  & & & \square\arrow[ruu] &  & \\
\end{tikzcd}
\begin{tikzcd}
 \square \ W_0\lp S,f \rp
&\triangle \ W_1\lp S,f \rp
&\diamond \ W_2\lp S,f \rp
&\star \ W_3\lp S,f \rp
&\times \ W_4\lp S, f \rp
\end{tikzcd}
\end{center}
\end{figure}

\section{Preliminaries}\label{laries}

We begin this section by noting that for certain parameters, we need only elementary tools to compute statistics of strictly preperiodic points.
For example, the fact that for any odd prime power $q$, the number of squares in $\F_q$ is $\frac{1}{2}\lp q+1\rp$ yields \cref{squaresremark}.
\begin{remark}\label{squaresremark}
Suppose $q$ is an odd prime power and $\alpha\in\F_q$.
Then
\[
w_0\lp \F_q, f_{2,\alpha} \rp = \frac{1}{2}\lp1-\frac{1}{q}\rp.
\]
\end{remark}
\noindent Of course, \cref{squaresremark} immediately generalizes to \cref{powersremark}.





\begin{proposition}\label{powersremark}
Let $q$ be a prime power and $\alpha\in\F_q$.
Then for any $d\in\Z_{\geq1}$,
\[
w_0\lp \F_q, f_{d,\alpha} \rp
=\lp1-\frac{1}{\gcd{\lp q-1,d\rp}}\rp\lp1-\frac{1}{q}\rp.
\]
\end{proposition}
\begin{proof}
Indeed, consider the bijection
\begin{align*}
\lp\F_q\rp^d&\mapsto f_{d,\alpha}\lp\F_q\rp\\
\beta&\mapsto\beta+\alpha,
\end{align*}
then use the fact that $\lp\F_q\rp^d$ has size $1+\frac{q-1}{\gcd{(q-1,d)}}$.

\end{proof}

The main technical tool we will use in proving our main results is \cref{technical}.

\begin{proposition}\label{technical}
Let $d\in\Z_{\geq2}$, and suppose $p$ is a prime that satisfies $p>(d!)^2$ and $p\equiv1\pmod{d}$.
Choose $r,n\in\Z_{\geq1}$.
If $r>2d^{2n}$, then for all $\alpha\in\F_{p^r}$ with $\F_p\lp \alpha \rp = \F_{p^r}$,
%
%
\[
\frac{2}{(d-1)(n+4+\log n)}-\frac{8d}{p^{r/2}}
<\frac{\lv f_{d,\alpha}^n\lp \F_{p^r} \rp\rv}{p^r}
<\frac{2}{(d-1)(n+1)}+\frac{8d}{p^{r/2}}.
\]
\end{proposition}
\begin{proof}
Let $R=\F_p[s]$ and $\phi(x)=x^d+s\in R[x]$.
We will apply \cite[Corollary~5.7]{G} to the dynamical system $(R,\phi)$.
To do so, we set $f(x)=\phi(x)-t\in R[t,x]$ and $K=\Frac(R[t])$, then write $L$ for the splitting field of $f(x)$ over $K$, write $B$ for the integral closure of $R[t]$ in $L$, write $G$ for $\Gal{(L/K)}$, and write $\rho$ for the action of $G$ on the roots of $f(x)$ in $B$.
Let $\pi(s)\in R$ be the minimal polynomial for $\alpha$ over $\F_p$, so that $\deg{\lp\pi(s)\rp}=r$ by hypothesis.
Since $p\equiv1\pmod{d}$, we see that $\Frac{\lp B/\pi(s)B\rp}/\Frac{\lp R[t]/\pi(s) R[t]\rp}$ is Galois with
\[
\Gal{\lp\Frac{\lp B/\pi(s)B\rp}/\Frac{\lp R[t]/\pi(s) R[t]\rp}\rp}
\simeq G\simeq\Z/d\Z.
\]
Moreover, as in the proof of \cite[Theorem~1.2]{G}, we know that $R/\pi(s)R$ is algebraically closed in $\Frac{\lp B/\pi(s)B\rp}$.


Let's write $S$ for the set of roots of $f(x)$ in $B$.
Let $[\rho]^n$ be the $n$th iterated wreath product of the action $\rho$; this is an action of the $n$th iterated wreath product of the group $G$ (denoted by $[G]^n$) on the set $S^n$ (see \cite[Section~5]{G} for more details).
Using this notation, let $\fix_n{\lp\rho\rp}$ be the proportion of $[G]^n$ with a fixed point under the action of $[\rho]^n$.
We are now in a position to apply \cite[Corollary~5.7]{G}.
Since $\phi$ is unicritical with critical point 0, \cite[Corollary~5.7]{G} holds for $n$ at most
\[
\left\lfloor\frac{\log\lp \log \lp p^r \rp \rp-\log{\lp\log{\lp p^2\rp}\rp}}{2\log d}\right\rfloor;
\]
this constraint follows by computing the height bound given in \cite[Definition~4.2]{G}, applied to the valuation on $\Frac{(R)}$ given by $\pi(s)$.
Since $\F_p(\alpha)=\F_{p^r}$, our hypothesis on $r=\deg{\lp\pi(s)\rp}$ ensures that $n$ satisfies this bound.
Therefore, noting that the specialization of $\phi$ at $\pi(s)R\in\Spec{(R)}$ is $f_{d,\alpha}$, we may apply \cite[Corollary~5.7]{G}.
However, since \cite[Corollary~5.7]{G} applies to $f_{d,\alpha}$ acting on $\P^1\lp\F_{p^r}\rp$, we must slightly adjust the constants appearing in the statement of that Corollary; using the inefficient estimate $1<dp^{r/d}(p^r+1)$, this adjustment yields
\[
\fix_n{(\rho)}-\frac{8d}{p^{r/2}}
<\frac{\lv f_{d,\alpha}^n\lp \F_{p^r} \rp\rv}{p^r}
<\fix_n{(\rho)}+\frac{8d}{p^{r/2}}.
\]
The result now follows by applying Juul's estimates on fixed point proportions in wreath products~\cite[Proposition~4.2]{JuulP}.
\end{proof}

\section{Effective upper bounds}\label{upboundz}

With \cref{technical} in hand, we proceed to proving the upper bounds on strictly preperiodic points mentioned in \cref{intro}.
Indeed, \cref{technical} immediately implies \cref{upperboundforw}.

\begin{proposition}\label{upperboundforw}
Let $d\in\Z_{\geq2}$, and suppose $p$ is a prime that satisfies $p>(d!)^2$ and $p\equiv1\pmod{d}$.
Choose $r,n\in\Z_{\geq1}$ and $\alpha\in\F_{p^r}$ with $\F_p\lp \alpha \rp = \F_{p^r}$.
If $r>2d^{2n+2}$, then 
\[
w_n\lp \F_{p^r}, f_{d,\alpha} \rp
<\frac{2\log{(n+1)}+8}{\lp d-1 \rp\lp n+1 \rp \lp n +5+\log{\lp n +1\rp}\rp}
+\frac{16d}{p^{r/2}}.
\]
\end{proposition}
\begin{proof}
\cref{technical} tells us that
\[
\frac{\left| f_{d,\alpha}^n\lp \F_{p^r} \rp \right|}{p^r}
< \frac{2}{\lp d-1 \rp\lp n+1 \rp} + \frac{8d}{p^{r/2}}
\hspace{15px}\text{and}\hspace{15px}
\frac{\left| f_{d,\alpha}^{n+1}\lp \F_{p^r} \rp \right|}{p^r}
>\frac{2}{\lp d-1 \rp\lp n+5+\log\lp n+1 \rp\rp}-\frac{8d}{p^{r/2}}.
\]
\end{proof}
\noindent We are now in a position to prove \cref{quadcor}, which we mentioned in \cref{intro}.
It is a simplification of the quadratic case of \cref{upperboundforw}.
(In \cref{orderofgrowthforW}, we present an even cruder simplification, which we will apply in our proof of \cref{averageupperbound}.)

\begin{proof}[Proof of \cref{quadcor}]
Since $2\leq n$, we know that $8<8\log{(n+1)}$, so that
\[
\frac{2\log{(n+1)}+8}{(n+1)(n+5+\log{(n+1)})}
<10\lp\frac{\log{(n+1)}}{n^2}\rp.
\]
Moreover, the fact that $3\leq n$ implies $n+1<n^{3/2}$, which tells us that
\[
10\lp\frac{\log{(n+1)}}{n^2}\rp+\frac{32}{p^{r/2}}
<10\lp\frac{\log{\lp n^{3/2}\rp}}{n^2}\rp+\frac{32}{p^{r/2}}
=15\lp\frac{\log{n}}{n^2}\rp+\frac{32}{p^{r/2}}.
\]
\end{proof}

\cref{technical} enables us to find upper bounds not just on sets of the form $W_n\lp\F_{p^r},f_{2,\alpha}\rp$, but also for the generalized sets $W_{m,n}\lp\F_{p^r},f_{d,\alpha}\rp$ for $d\in\Z_{\geq2}$.

\begin{theorem}\label{upperboundforW}
Let $d\in\Z_{\geq2}$, and suppose $p$ is a prime that satisfies $p>(d!)^2$ and $p\equiv1\pmod{d}$.
Choose $r,m,n\in\Z_{\geq1}$ with $1<m<n$, and $\alpha\in\F_{p^r}$ with $\F_p\lp \alpha \rp = \F_{p^r}$.
If $r>2d^{2n}$, then
\[
w_{m,n}\lp \F_{p^r},f_{d,\alpha}\rp
<\frac{2}{d-1} \lp \frac{1}{m}-\frac{1}{n}
+\frac{4\log{n}}{mn} \rp +\frac{16d}{p^{r/2}}.
\]
\end{theorem}
\begin{proof}
Using \cref{technical} as in \cref{upperboundforw}, we see that
\begin{align*}
&w_{m,n}\lp \F_{p^r},f_{d,\alpha}\rp\\
&\qquad<\lp \frac{2}{\lp d-1 \rp\lp m+1 \rp}
+\frac{8d}{p^{r/2}} \rp
- \lp \frac{2}{\lp d-1\rp \lp n+4+\log n \rp}
-\frac{8d}{p^{r/2}} \rp\\
&\qquad= \frac{2n-2m +2\log n +6}{\lp d-1 \rp\lp m+1\rp\lp n+4+\log n \rp}+\frac{16d}{p^{r/2}}\\
&\qquad<\frac{2n - 2m + 8\log{n}}{(d-1) mn}+\frac{16d}{p^r}
&&\text{(since $6<6\log{n}$)}\\
&\qquad= \frac{2}{d-1} \lp \frac{1}{m}-\frac{1}{n} +\frac{4\log n}{mn} \rp +\frac{16d}{p^{r/2}}.
\end{align*}
\end{proof}
\noindent We remark that \cref{upperboundforW} establishes one half of \cref{aaroncorollary}.

\section{Effective lower bounds}
\label{theorems}

We proceed to proving the lower bound of \cref{aaroncorollary}.

\begin{proposition}\label{lowerboundforW}
Let $d\in\Z_{\geq2}$, and suppose $p$ is a prime that satisfies $p>(d!)^2$ and $p\equiv1\pmod{d}$.
Choose $r,m,n\in\Z_{\geq1}$ with $5<m<n$, and $\alpha\in\F_{p^r}$ with $\F_p\lp \alpha \rp = \F_{p^r}$.
If $r>2d^{2n}$, then
\[
w_{m,n}\lp \F_{p^r},f_{d,\alpha}\rp
>\frac{7}{8(d-1)}\lp\frac{1}{m}-\frac{1}{n}-\frac{4\log{m}}{mn}\rp-\frac{16d}{p^{r/2}}.
\]
\end{proposition}
\begin{proof}
Apply \cref{technical} to see that
\begin{align*}
&w_{m,n}\lp \F_{p^r},f_{d,\alpha}\rp\\
&\qquad>\lp \frac{2}{\lp d-1 \rp\lp m+4+\log{m} \rp}
-\frac{8d}{p^{r/2}} \rp
- \lp \frac{2}{\lp d-1\rp \lp n+1 \rp}
+\frac{8d}{p^{r/2}} \rp\\
&\qquad=\frac{2n-2m -2\log m -6}{\lp d-1 \rp\lp n+1\rp\lp m+4+\log m \rp}
-\frac{16d}{p^{r/2}}\\
&\qquad>\frac{2n-2m-8\log m}{(d-1)(n+1)(m+4+\log m)}-\frac{16d}{p^r}
&&\text{(since $6<6\log{m}$)}.
\end{align*}
Since $5<m<n$, we observe
\[
\frac{mn}{(n+1)(m+4+\log m)} > \frac{mn}{(n+1)(2m)}\geq\frac{7}{16}.
\]
Thus, we see that
\[
\frac{2n-2m-8\log m}{(d-1)(n+1)(m+4+\log m)}-\frac{16d}{p^{r/2}}
> \frac{7}{16(d-1)} \lp \frac{2n-2m-8\log m}{mn} \rp - \frac{16d}{p^{r/2}}.
\]
\end{proof}
\noindent\cref{aaroncorollary} now follows immediately from \cref{upperboundforW} and \cref{lowerboundforW}.



%

\section{Averaging over polynomials}
\label{averaging}

We now compute statistics of our strictly preperiodic partitions over all quadratic polynomials.
We first prove \cref{orderofgrowthforW}, which is a simplification of \cref{upperboundforw}.
We use this simplification only to aid our proof of \cref{averageupperbound}.

\begin{corollary}\label{orderofgrowthforW}
Keep the hypotheses of \cref{upperboundforw}.
If $n>133$, then
\[
w_n\lp \F_{p^r}, f_{d,\alpha}\rp
<\frac{1}{n^{3/2}}+\frac{16d}{p^{r/2}}.
\]
\end{corollary}
\begin{proof}
Indeed, for all such $n$,
\[
\frac{2\log\lp n+1\rp +8}
{\lp d-1 \rp\lp n+1\rp\lp n+5+\log\lp n +1\rp\rp}
< \frac{1}{n^{3/2}}.
\]
The result now follows from \cref{upperboundforw}.
\end{proof}

\cref{orderofgrowthforW} in hand, we now prove \cref{averageupperbound}.


\begin{proof}[Proof of \cref{averageupperbound}]
We begin by counting the number of quadratic polynomials that are conjugate to a given unicritical polynomial.
To this end, let's write
\[
\mathcal{Q}=\lb f\in\F_{p^r}[x]\mid\deg{\lp f \rp} = 2\rb
\qquad\text{and}\qquad
\mathcal{U}=\lb x^2+\delta\mid\delta\in\F_{p^r}\rb.
\]
Since $p$ is odd, for any $\alpha\in\F_{q^r}\setminus\lb0\rb$ and $\beta\in\F_{q^r}$ we may define the following coordinate change on $\F_{p^r}$:
\[
\mu_{\alpha,\beta}:X\mapsto\alpha X+\frac{\beta}{2}.
\]
Next, we set
\begin{align*}
\mu\colon\hspace{70px}\mathcal{Q}&\to\mathcal{U}\\
\alpha X^2+\beta X+\gamma&\mapsto X^2-\frac{\beta^2-4\alpha\gamma-2\beta}{4}.
\end{align*}
Then $\mu$ is surjective and $p^r\lp p^r-1\rp$-to-one.
Moreover, for any $f\in\mathcal{Q}$, say with $f(X)=\alpha X^2+\beta X+\gamma$, we see that
\[
\mu(f)=\mu_{\alpha,\beta}\circ f\circ\mu_{\alpha,\beta}^{-1};
\]
thus,
\[
\lv W_n\lp\F_{p^r},f\rp\rv
=\lv W_n\lp\F_{p^r},\mu(f)\rp\rv.
\]

Let's write $\lp \F_{p^r}\rp^{\text{prim}}$ for the set of $\alpha\in\F_{p^r}$ with $\F_{p}\lp \alpha \rp = \F_{p^r}$, and recall $\lv\F_{p^r}\setminus\lp \F_{p^r}\rp^{\text{prim}}\rv<2p^{r/2}$.
Then by the first paragraph of this proof, we see
\begin{align*}
&\frac{1}{|\mathcal{Q}|}
\sum_{f\in \mathcal{Q}}
w_n\lp \F_{p^r}, f\rp\\
&\qquad=\frac{p^r\lp p^r-1\rp}{p^{3r}-p^{2r}}
\sum_{f\in \mathcal{U}}
w_n\lp \F_{p^r}, f\rp\\
&\qquad<\frac{1}{p^r}
\lp\sum_{\delta\in \lp \F_{p^r}\rp^{\text{prim}}}
w_n\lp \F_{p^r}, x^2+\delta\rp
+\sum_{\delta \in\F_{p^r}\setminus\lp \F_{p^r}\rp^{\text{prim}}}
w_n\lp \F_{p^r}, x^2+\delta\rp \rp\\
&\qquad<\frac{1}{p^r}\lp p^r\lp\frac{1}{n^{3/2}}+\frac{32}{p^{r/2}} \rp+2p^{r/2}\rp
&&\hspace{-30px}\text{(by \cref{orderofgrowthforW})}\\
&\qquad=\frac{1}{n^{3/2}}+\frac{34}{p^{r/2}}.
\end{align*}
\end{proof}
\cref{averageupperbound} applies when $n>133$.
If we are willing to accept a higher threshold for $n$, we achieve \cref{strongest}, a stronger bound.

\begin{theorem}\label{strongest}
Suppose $p>3$ is prime.
Let $\epsilon\in \R_{>0}$ and $n,r\in\Z_{\geq 1}$.
Then there exists $N_{\epsilon}\in\Z_{>0}$ such that if $n>N_{\epsilon}$ and $r>2^{2n+3}$, then
\[
\frac{1}
{\lv\lb f\in\F_{p^r}[x]
\mid\deg{f}=2\rb\rv}
\cdot\sum_{\substack{f\in\F_{p^r}[x]\\
\deg{f}=2}}
{w_n\lp \F_{p^r}, f\rp}
<\frac{1}{n^{2-\epsilon}}
+\frac{34}{p^{r/2}}.   
\]
\end{theorem}
\begin{proof}
Choose $N_{\epsilon}$ so that for any $n>N_{\epsilon}$
\[
\frac{2\log(n+1)+8}{(n+1)(n+5+\log(n+1))}
< \frac{1}{n^{2-\epsilon}}.
\]
The remainder of the proof is similar to that of \cref{averageupperbound}.
\end{proof}

Using \cref{upperboundforW} instead of \cref{orderofgrowthforW}, we prove \cref{genhead}.

\begin{proposition}\label{genhead}
Suppose $p>3$ is prime.
Let $r,m,n\in\Z_{\geq 1}$ with $1<m<n$.
If $r>2^{2n+1}$, then
\[
\frac{1}
{\lv\lb f\in\F_{p^r}[x]
\mid\deg{f}=2\rb\rv}
\cdot\sum_{\substack{f\in\F_{p^r}[x]\\
\deg{f}=2}}
{w_{m,n}\lp \F_{p^r}, f\rp}
<2\lp\frac{1}{m} - \frac{1}{n}\rp
+9\lp\frac{\log n}{mn}\rp.
\]
\end{proposition}
\begin{proof}
Keeping the same notation as the proof of \cref{averageupperbound}, note that
\begin{align*}
&\frac{1}{|\mathcal{Q}|}\sum_{f\in \mathcal{Q}}w_{m,n}\lp \F_{p^r}, f\rp\\
&\qquad<\frac{p^r\lp p^r-1\rp}{p^{3r}-p^{2r}}
\lp \sum_{\alpha\in \lp \F_{p^r}\rp^{\text{prim}}}
{w_{m,n}\lp \F_{p^r}, x^2+\alpha\rp}
+ \sum_{\alpha\in\F_{p^r}\setminus\lp \F_{p^r}\rp^{\text{prim}}}
{w_{m,n}\lp \F_{p^r}, x^2+\alpha\rp} \rp\\
&\qquad< \frac{1}{p^r}
\lp p^r\lp \frac{2}{m}-\frac{2}{n}+\frac{8\log{n}}{mn}+\frac{32}{p^{r/2}} \rp
+2p^{r/2}\rp
&&\hspace{-50px}\text{(by \cref{upperboundforW})}\\
&\qquad=\frac{2}{m}-\frac{2}{n} +\frac{8\log n}{mn}+\frac{34}{p^{r/2}}.
\end{align*}
And since $r>2^{2n+1}$, we conclude by noting that
\[
\frac{34}{p^{r/2}}<\frac{\log{n}}{mn}.
\]
    \end{proof}

Finally, we prove \cref{averagelowerboundmn}, completing the proof of \cref{genheadintro}.

\begin{proposition}\label{averagelowerboundmn}
Keep the hypotheses of \cref{genhead}, but assume $5<m$.
Then
\[
\frac{1}
{\lv\lb f\in\F_{p^r}[x]
\mid\deg{f}=2\rb\rv}
\cdot\sum_{\substack{f\in\F_{p^r}[x]\\
\deg{f}=2}}
{w_{m,n}\lp \F_{p^r}, f\rp}
> \frac{7}{8}\lp\frac{1}{m} - \frac{1}{n} \rp-4\lp\frac{\log m}{mn}\rp.
\]
\end{proposition}
\begin{proof}
This follows by a similar argument to \cref{genhead}, using \cref{lowerboundforW} instead of \cref{upperboundforW}.
\end{proof}

\section*{Acknowledgements} \label{Acknowledgements}

We would very much like to thank John Caughman, who asked the question that led to this paper.
We would also like to thank the anonymous reviewer for many useful comments.

\bibliography{headwaters}
\bibliographystyle{amsalpha}

\end{document}